\documentclass[11pt,leqno]{article}
\usepackage[doc]{optional}
\usepackage{enumitem}

\usepackage{wrapfig}
\usepackage{color}
\usepackage{float}
\usepackage{soul}
\usepackage{graphicx}
\usepackage{graphics}
\definecolor{labelkey}{rgb}{0,0.08,0.45}
\definecolor{refkey}{rgb}{0,0.6,0.0}

\definecolor{Brown}{rgb}{0.45,0.0,0.05}
\definecolor{lime}{rgb}{0.00,0.8,0.0}
\definecolor{lblue}{rgb}{0.5,0.5,0.99}
\definecolor{lblue}{rgb}{0.8,0.85,1.00}
\definecolor{anotherblue}{rgb}{.8, .8,1}
\definecolor{violet}{rgb}{0.9,0.6,0.9}
\definecolor{greenyellow}{rgb}{0.53,0.99,0.18}

\definecolor{Lyellow}{rgb}{0.87,0.87,0.87}
\definecolor{Lgray}{rgb}{0.92,0.92,0.92}
\definecolor{Mgray}{rgb}{0.5,0.5,0.5}
\definecolor{Gold}{rgb}{0.99,0.84,0.0}

\definecolor{myblue}{rgb}{.8, .8, 1}
  \newcommand*\mybluebox[1]{%
    \colorbox{myblue}{\hspace{1em}#1\hspace{1em}}}

\usepackage{mathpazo}
\usepackage{amsmath}
\usepackage{amssymb}
\usepackage{theorem}
\usepackage{fancyhdr}
\usepackage{empheq}

\usepackage[margin=1.0in]{geometry}

\usepackage{hyperref}

\newcommand{\weakly}{\ensuremath{\:{\rightharpoonup}\:}}

\newcommand{\frechet}{Fr\'echet}
\newcommand{\gateaux}{G\^ateaux}
\newcommand{\nnn}{\ensuremath{{n\in{\mathbb N}}}}

\newcommand{\thalb}{\ensuremath{\tfrac{1}{2}}}
\newcommand{\menge}[2]{\big\{{#1}~\big |~{#2}\big\}}

\newcommand{\To}{\ensuremath{\rightrightarrows}}

\newcommand{\fenv}[1]%
{\ensuremath{\,\overrightarrow{\operatorname{env}}_{#1}}}
\newcommand{\benv}[1]%
{\ensuremath{\,\overleftarrow{\operatorname{env}}_{#1}}}

\newcommand{\scal}[2]{\left\langle{#1},{#2}  \right\rangle}

\newcommand{\GG}{\ensuremath{\mathcal G}}

\newcommand{\zeroun}{\ensuremath{\left]0,1\right[}}
\newcommand{\RR}{\ensuremath{\mathbb R}}

\newcommand{\NN}{\ensuremath{\mathbb N}}

\newcommand{\ball}{\ensuremath{\mathrm{ball}}}

\newcommand{\argmin}{\ensuremath{\operatorname{argmin}}}

\newcommand{\prox}{\ensuremath{\operatorname{P}}}

\newcommand{\bd}{\ensuremath{\operatorname{bdry}}}

\newcommand{\ran}{\ensuremath{\operatorname{ran}}}

\newcommand{\conv}{\ensuremath{\operatorname{conv}}}
\newcommand{\sgn}{\ensuremath{\operatorname{sgn}}}
\newcommand{\cone}{\ensuremath{\operatorname{cone}}}

\newcommand{\Fix}{\ensuremath{\operatorname{Fix}}}

\newcommand{\reck}{\ensuremath{\operatorname{rec}}}
\newcommand{\Id}{\ensuremath{\operatorname{Id}}}

\newcommand{\pinf}{\ensuremath{+\infty}}

{\begin{list}{}{%
\settowidth{\labelwidth}{\textrm{#1~}}%
\setlength{\leftmargin}{\labelwidth+\labelsep}}}
{\end{list}}
\newtheorem{theorem}{Theorem}[section]
\newtheorem{lemma}[theorem]{Lemma}
\newtheorem{corollary}[theorem]{Corollary}
\newtheorem{proposition}[theorem]{Proposition}

\theoremstyle{plain}{\theorembodyfont{\rmfamily}
}
\theoremstyle{plain}{\theorembodyfont{\rmfamily}
}
\theoremstyle{plain}{\theorembodyfont{\rmfamily}
}
\theoremstyle{plain}{\theorembodyfont{\rmfamily}
\newtheorem{example}[theorem]{Example}}
\newtheorem{fact}[theorem]{Fact}
\theoremstyle{plain}{\theorembodyfont{\rmfamily}
\newtheorem{remark}[theorem]{Remark}}

\def\doi{DOI}

\newcounter{count}

\allowdisplaybreaks

\begin{document}

\title{\textrm{On subgradient projectors}}

\author{
Heinz H.\ Bauschke\thanks{Mathematics, University of British
Columbia, Kelowna, B.C.\ V1V~1V7, Canada. E-mail:
\texttt{heinz.bauschke@ubc.ca}.},~
Caifang Wang\thanks{Department of Mathematics,
Shanghai Maritime University, China. E-mail: \texttt{cfwang@shmtu.edu.cn}.},~Xianfu Wang\thanks{Mathematics, University of
 British Columbia, Kelowna, B.C.\ V1V~1V7, Canada. E-mail: \texttt{shawn.wang@ubc.ca}.},~
 and
Jia Xu\thanks{Mathematics, University of British Columbia,
Kelowna, B.C.\ V1V~1V7, Canada. E-mail:
\texttt{jia.xu@ubc.ca}.}}

\date{March 27, 2014}

\maketitle \thispagestyle{fancy}

\vskip 8mm

\begin{abstract} \noindent
The subgradient projector is of considerable importance in convex
optimization because it plays the key role in
Polyak's seminal work --- and the many papers it spawned --- 
on subgradient projection algorithms 
for solving convex feasibility problems. 

In this paper, we offer a systematic study of the subgradient
projector. Fundamental properties such as continuity, nonexpansiveness,
and monotonicity are investigated. We also discuss the
Yamagishi--Yamada operator. Numerous examples illustrate our results.
\end{abstract}


\frechet\ differentiability,
\gateaux\ differentiability, 
monotone operator,
nonexpansive mapping,
subgradient projector,
Yamagishi--Yamada operator. 


\section{Introduction}

Throughout this paper,
we assume that
\begin{empheq}[box=\mybluebox]{equation}
\text{$X$ is a real Hilbert space}
\end{empheq}
with inner product $\scal{\cdot}{\cdot}$ and induced norm
$\|\cdot\|$. 
We also assume that 
\begin{empheq}[box=\mybluebox]{equation}
\text{$f\colon X\to\RR$ is convex and continuous, and $C = \menge{x\in
X}{f(x)\leq 0}\neq\varnothing$. }
\end{empheq}
(When $X$ is finite-dimensional,
we do not need to explicitly impose continuity on $f$.)
Unless stated otherwise, we assume that $s\colon X\to X$ 
is a \emph{selection}
of $\partial f$, i.e.,
\begin{empheq}[box=\mybluebox]{equation}
(\forall x\in X)\quad s(x)\in \partial f(x)
\end{empheq}
and that 
$G\colon X\to X$ is the \emph{associated subgradient projector} defined by 
\begin{empheq}[box=\mybluebox]{equation}
(\forall x\in X)\quad 
Gx = 
\begin{cases}
\displaystyle x - \frac{f(x)}{\|s(x)\|^2}s(x), &\text{if $f(x)>0$;}\\
x, &\text{otherwise.}
\end{cases}
\end{empheq}
Observe this is well defined because $C\neq\varnothing$ and thus
$0\notin\partial f(X\smallsetminus C)$. 

When we need to exhibit the underlying function $f$ or
subgradient selection $s$, we shall write
$s_f$, $C_f$ and $G_f=G_{f,s}$ instead of 
$s$, $C$ and $G$, respectively. 

The subgradient projector is the key ingredient in 
Polyak's seminal work \cite{Poljak} on subgradient projection 
algorithms\,\footnote{See also \cite{Goffin} for a historical
account.},
which have since found many applications; see, e.g.,
\cite{bb96}, \cite{MOR}, \cite{Cegielski}, \cite{CL},
\cite{CS}, \cite{CenZen}, \cite{Comb93}, \cite{Comb97},
\cite{CombLuo}, 
\cite{Polyakbook}, \cite{PolyakHaifa}, 
\cite{SY}, \cite{YO1}, \cite{YO2}, \cite{YSY},
and the references therein. 

{\em 
The aim of this paper is to provide a systematic 
study of the subgradient operator.
We review known properties, present basic calculus rules,
obtain characterization of strong-to-strong and strong-to-weak
continuity, analyze nonexpansiveness, monotonicity, and
the decreasing property, and discuss the relationship to the 
Yamagishi--Yamada operator. Numerous examples illustrate our
results.
}

The paper is organized as follows. 
Basic properties are reviewed in Section~\ref{s:prelim},
and basic calculus rules are derived in Section~\ref{s:calc}.
Section~\ref{s:examples} is a collection of examples.
The relationship between 
strong-to-strong (resp.\ strong-to-weak) continuity of $G$
and \frechet\ (resp.\ \gateaux) differentiability of $f$ is
clarified in Section~\ref{s:contF} (resp.\ Section~\ref{s:contG}). 
The case when $f$ arises from a quadratic form is investigated in
Section~\ref{s:Frank}. Nonexpansiveness and the decreasing
property are studied in Section~\ref{s:nonexp} and
\ref{s:decrease}, respectively. 
These properties are illustrated with in Section~\ref{s:pnorm}.
In the final Section~\ref{s:YY}, 
we provide a sufficient condition for the  Yamagishi--Yamada
operator to be itself a subgradient projector. 

Notation and terminology are standard and follow largely
\cite{BC2011} to which we refer the reader if needed. 
We do write $\prox_f = (\Id+\partial f)^{-1}$ 
for the proximity operator (proximal mapping) of $f$.

\section{Preliminary results}

\label{s:prelim}

Let us record some basic results on subgradient projectors, which are
essentially contained already in \cite{Poljak} and the
proofs of which we provide for completeness.

\begin{fact}
\label{f:known}
Let $x\in X$, and set
\begin{equation}
H = \menge{y\in X}{\scal{s(x)}{y-x} + f(x) \leq 0}.
\end{equation}
Then the following hold:
\begin{enumerate}
\item 
\label{f:known0}
$f^+(x)+\scal{s(x)}{Gx-x}=0$.
\item 
\label{f:known1}
$\Fix G = C\subseteq H$. 
\item 
\label{f:known2}
$Gx=P_Hx$. 
\item 
\label{f:known3}
$(\forall c\in C)$ 
$\scal{c-Gx}{x-Gx}\leq 0$. 
\item 
\label{f:known3+}
$(\forall c\in C)$ 
$\|x-Gx\|^2 + \|Gx-c\|^2 \leq \|x-c\|^2$.
\item 
\label{f:known4}
$f^+(x) = \|s(x)\|\|x-Gx\|$.
\item 
\label{f:known4+}
If $x\notin C$, then $(\forall c\in C)$ 
$f^2(x)\|s(x)\|^{-2}+\|Gx-c\|^2 \leq \|x-c\|^2$. 
\item 
\label{f:known5}
$f^+(x)(x-Gx) = \|x-Gx\|^2s(x)$. 
\item
\label{f:known6}
Suppose that $f$ is \frechet\ differentiable at $x\in
X\smallsetminus C$.
Then $g = \ln\circ f\colon X\smallsetminus C \to\RR$ 
is \frechet\ differentiable at $x$ and
$Gx = x - \nabla g(x)/\|\nabla g(x)\|^2$. 
\item
\label{f:known7}
Suppose that $\min f(X)=0$, that $f$ is \frechet\ differentiable on $X$ with
$\nabla f$ being Lipschitz continuous with constant $L$,
that $x\notin C$, 
and that there exists $\alpha>0$ such that 
$f(x)\geq \alpha d_C^2(x)$. Then
$d^2_C(Gx) \leq (1-\alpha^2/L^2)d_C^2(x)$. 
\item
\label{f:known8}
Suppose that $\min f(X)=0$, that $x\notin C$, 
and that there exists $\alpha>0$ such that 
$f(x)\geq \alpha d_C(x)$. Then
$d^2_C(Gx) \leq (1-\alpha^2/\|s(x)\|^2)d_C^2(x)$. 
\end{enumerate}
\end{fact}
\begin{proof}
Let $z\in X$.
\ref{f:known0}: This follows directly from the definition of $G$. 

\ref{f:known1}: 
The equality is clear from the definition of $G$. 
Assume that $z\in C$. Then
$\scal{s(x)}{z-x} + f(x) \leq f(z) \leq 0$ and hence $z\in H$. 

\ref{f:known2}: 
Assume first that $x\in C$.
Then $x\in\Fix G\subseteq H$ by \ref{f:known1}
and hence $Gx=x= P_Hx$. 
Now assume that $x\notin C$. 
Then $0<f(x)=f^+(x)$ and $s(x)\neq 0$. 
Hence, 
\begin{equation}
P_Hx = x - \frac{\Big(\scal{s(x)}{x}-
\big( \scal{s(x)}{x}-f(x)\big)\Big)^+}{\|s(x)\|^2}s(x)
= x - \frac{f^+(x)}{\|s(x)\|^2}s(x) = Gx. 
\end{equation}

\ref{f:known3}: 
In view of \ref{f:known2}, we have
$(\forall h\in H)$ $\scal{h-Gx}{x-Gx} \leq 0$. 
Now invoke \ref{f:known1}. 

\ref{f:known3+}: 
This is equivalent to \ref{f:known3}. 

\ref{f:known4}: 
Assume first that $x\in C$.
Then $f(x)\leq 0$, i.e., $f^+(x)=0$, and $x=Gx$ by \ref{f:known1}.
Hence the identity is true.
Now assume that $x\notin C$. 
Then $0<f(x)=f^+(x)$ and $x-Gx = f(x)/\|s(x)\|^2 s(x)$. 
Taking the norm, we learn that
$\|x-Gx\| = f(x)/\|s(x)\|=f^+(x)/\|s(x)\|$. 

\ref{f:known4+}: Combine \ref{f:known1}, \ref{f:known3+}, and
\ref{f:known4}. 

\ref{f:known5}: 
This follows from \ref{f:known4} and the definition of $G$. 

\ref{f:known6}: The chain rule implies that
$\nabla g(x) = (1/f(x))\nabla f(x)$. 
Hence $\|\nabla g(x)\|^2 = \|\nabla f(x)\|^2/f^2(x)$ and thus
$x-\nabla g(x)/\|\nabla g(x)\|^2 = x - f(x)/\|\nabla f(x)\|^2 \nabla
f(x) = Gx$. 

\ref{f:known7}: 
Let $c\in C$. Then $\nabla f(c)=0$
and hence $\|\nabla f(x)\| = 
\|\nabla f(x)-\nabla f(c)\|\leq L\|x-c\|$.
Hence $\|\nabla f(x)\|\leq Ld_C(x)$ and therefore,
using \ref{f:known4+}, we obtain 
\begin{equation}
\|Gx-c\|^2 \leq \|x-c\|^2 - \frac{f^2(x)}{\|\nabla f(x)\|^2}
\leq \|x-c\|^2 - \frac{\alpha^2d_C^4(x)}{L^2d_C^2(x)}. 
\end{equation}
Now take the minimum over $c\in C$. 

\ref{f:known8}: 
Using \ref{f:known4+}, we have
\begin{equation}
d^2_C(Gx) \leq \|Gx-P_Cx\|^2 \leq \|x-P_Cx\|^2 - \frac{f^2(x)}{\|s(x)\|^2}
\leq d_C^2(x) - \frac{\alpha^2d_C^2(x)}{\|s(x)\|^2}.
\end{equation}
The proof is complete.
\end{proof}

\section{Calculus}

\label{s:calc}

We now turn to basic calculus rules.
When the proof is a straight-forward verification, we will omit
it. It is convenient to introduce the operator
$\GG \colon X\To X$, defined by
\begin{equation}
(\forall x\in X)\quad
\GG x = \GG_f x = \menge{G_s(x)}{\text{$s$ is a selection of $\partial
f$}}.
\end{equation}
When $G$ is \gateaux\ differentiable outside $C$, then we will
identify $\GG$ with $G$. 

\begin{proposition}[calculus]
\label{p:calc}
Let $\alpha>0$, let $A\colon X\to X$ be continuous and linear
such that $A^*A = A^*A = \Id$, and let $z\in X$.
Furthermore, let $(f_i)_{i\in I}$ be a finite family of convex
continuous functions on $X$ such that $\bigcap_{i\in I}
C_{f_i}\neq\varnothing$.
Then the following hold:
\begin{enumerate}
\item
\label{p:calc1}
Suppose that $g = \alpha f$. 
Then $C_{g} = C_f$ and $\GG_g = \GG_f$. 
\item
\label{p:calc2}
Suppose that $g =  f\circ \alpha\Id$. 
Then $C_{g} = \alpha^{-1}C_f$ and 
$\GG_g = \alpha^{-1}\GG_f\circ \alpha\Id$. 
\item
\label{p:calc3}
Suppose that $f\geq 0$ and that $g = f^\alpha$ is convex.
Then $C_{g} = C_f$ and $\GG_g = (1-\alpha^{-1})\Id +
\alpha^{-1}\GG_f$.
\item
\label{p:calc4}
Suppose that $g = f \circ A$. 
Then
$C_g = A^*C_f$ and $\GG_g = A^*\circ \GG_f \circ A$. 
\item
\label{p:calc5}
Suppose that $g \colon x\mapsto f(x-z)$.
Then
$C_g = z+C_f$ and $\GG_g\colon x\mapsto z+\GG_f(x-z)$. 
\item
\label{p:calc6}
Suppose that $g=\max_{i\in I} f_i$.
Then $C_g = \bigcap_{i\in I} C_{f_i}$ and
if $g(x)>0$ and $I(x) = \menge{i\in I}{f_i(x)=g(x)}$, then 
$\GG_g(x) = \menge{x - g(x)\|x^*\|^{-2}x^*}{x^*\in\conv \bigcup_{i\in I(x)}\partial f_i(x)}$. 
\item
\label{p:calc7}
Suppose that $g = f^+$. 
Then $\GG_g = \GG_f$. 
\item
\label{p:calc8}
\textbf{(Moreau envelope)}
Suppose that $\min f(X)=0$ and that
$g = f\Box (1/2)\|\cdot\|^2$ is the Moreau envelope of $f$.
Then $C_g = C_f$ and 
\begin{equation}
(\forall x\in X)\quad
G_g(x) = \begin{cases}
x - \displaystyle 
\frac{g(x)}{\|x-\prox_fx\|^2}(x-\prox_fx), &\text{if $f(x)>0$;}\\
x, &\text{if $f(x)=0$.}
\end{cases}
\end{equation}
\end{enumerate}
\end{proposition}
\begin{proof}
Let $x\in X$.
We shall only prove one inclusion for the subgradient projector
as the remaining one is proved similarly.

\ref{p:calc1}:
Since $g(x)\leq 0$ 
$\Leftrightarrow$
$f(x)\leq 0$, it follows that $C_g = C_f$.
Suppose that $f(x)>0$. 
Since $\alpha s_f(x)\in \partial g(x)$, we obtain 
$G_fx= x - f(x)\|s_f(x)\|^{-2}s_f(x)
=x -g(x)/\|\alpha s_f(x)\|^{-2}(\alpha s_f(x))$. This
implies $\GG_f(x)\subseteq \GG_g(x)$.

\ref{p:calc2}: Suppose that $g(x)>0$, i.e., $f(\alpha x)>0$.
Then
$\alpha^{-1}G_f(\alpha x) = 
\alpha^{-1}(\alpha x - f(\alpha x)\|s_f(\alpha x)\|^{-2}s_f(\alpha
x))
= x - \alpha^{-1}f(\alpha x)\|s_f(\alpha x)\|^{-2}s_f(\alpha x)
= x - f(\alpha x)\|\alpha s_f(\alpha x)\|^{-2}(\alpha s_f(\alpha x))
\in \GG_g(x)$. Hence $\alpha^{-1}\GG_f(\alpha x)\subseteq
\GG_g(x)$.

\ref{p:calc3}:
Suppose that $g(x)>0$.
Then $f(x)>0$ and 
$\alpha^{-1}(x-G_fx) = 
\alpha^{-1}f(x)/\|s_f(x)\|^2 s_f(x)
= f^\alpha(x)\|\alpha
f^{\alpha-1}(x)s_f(x)\|^{-2}\alpha
f^{\alpha-1}(x)s_f(x)\in x-\GG_g(x)$.

\ref{p:calc4}:
We have $x\in C_g$
$\Leftrightarrow$ 
$f(Ax)\leq 0$
$\Leftrightarrow$ 
$Ax\in C_f$
$\Leftrightarrow$ 
$x\in A^*C_f$.
Suppose that $g(x)>0$.
Then $f(Ax)>0$, $A^*s_f(Ax)\in\partial g(x)$ and 
$A^*G_f(Ax) =
A^*(Ax - f(Ax)\|s_f(Ax)\|^{-2}s_f(Ax))
= x - f(Ax)\|A^*s_f(Ax)\|^{-2}A^*s_f(Ax)
\in\GG_g(x)$. 

\ref{p:calc5}:
Suppose that $0<g(x)=f(x-z)$.
Then 
$z+G_f(x-z) = 
z+ x-z - f(x-z)\|s_f(x-z)\|^{-2}s_f(x-z)
\in \GG_g(x)$. 

\ref{p:calc6}:
This follows from the well known formula for the subdifferential
of a maximum; see, e.g., \cite[Proposition~3.38]{Penot}. 

\ref{p:calc7}: This follows from \ref{p:calc6} since $f^+ =
\max\{0,f\}$. 

\ref{p:calc8}:
This is clear because $g\geq 0$,
$\nabla g = \Id-\prox_f$
(see, e.g., \cite[Proposition~12.29]{BC2011}), and
$\argmin g = \argmin f$ (see, e.g., \cite[Corollary~17.5]{BC2011}). 
\end{proof}

\section{Examples}

\label{s:examples}

In this section, we present several illustrative examples. 

\begin{example}
Suppose that $f = \|\cdot\|^2$. Then $\nabla f = 2\Id$ and 
$G = \thalb\Id$. 
\end{example}

\begin{example}
\label{ex:huber}
Suppose that 
\begin{equation}
(\forall x\in X)\quad
f(x) = \begin{cases}
\thalb \|x\|^2, &\text{if $x\in\ball(0;1)$;}\\
\|x\|-\thalb, &\text{otherwise.}
\end{cases}
\end{equation}
Then $G = \thalb P_{\ball(0;1)}$ and $G$ is firmly nonexpansive.
\end{example}
\begin{proof}
Let $x\in X$.
Observe that $f = \|\cdot\|\Box(1/2)\|\cdot\|^2$ is
the Moreau envelope of the norm. 
Hence it follows from Proposition~\ref{p:calc}\ref{p:calc8} that
\begin{equation}
Gx = x - \frac{f(x)}{\|x-\prox_{\|\cdot\|}x\|^2}(x-\prox_{\|\cdot\|}x)
\end{equation}
provided that $x\neq 0$, and $Gx=0=\thalb P_{\ball(0;1)}x$ if $x=0$.
Furthermore, $\prox_{\|\cdot\|} =
\Id - \prox_{\|\cdot\|^*} =
\Id - \prox_{\iota_{\ball(0;1)}}
= \Id-P_{\ball(0;1)}$.
Thus, $\Id-\prox_{\|\cdot\|} = P_{\ball(0;1)}$. 
Assume now $x\neq 0$.
If $0<\|x\|\leq 1$, then 
\begin{equation}
Gx = x - \frac{\thalb\|x\|^2}{\|P_{\ball(0;1)}x\|^2}P_{\ball(0;1)}(x)
= x - \frac{\|x\|^2}{2\|x\|^2}x = \thalb x = \thalb P_{\ball(0;1)}x;
\end{equation}
and if $1<\|x\|$, then 
\begin{equation}
Gx = x - \frac{\|x\|-\thalb}{\|P_{\ball(0;1)}x\|^2}P_{\ball(0;1)}(x)
= x - \frac{\|x\|-\thalb}{\big\|x/\|x\|\big\|^2}\frac{x}{\|x\|} = \thalb
\frac{x}{\|x\|} = \thalb P_{\ball(0;1)}x.
\end{equation}
Now $P_{\ball(0;1)}$ is firmly nonexpansive, and hence so is
$\Id-P_{\ball(0;1)}$. It follows that 
$2G-\Id = -(\Id-P_{\ball(0;1)})$ is
nonexpansive, and therefore that $G$ is firmly nonexpansive. 
\end{proof}

\begin{proposition}
\label{p:maxdist}
Let $(C_i)_{i\in I}$ be a finite family of 
closed convex subsets of $X$ such that 
$C = \bigcap_{i\in I} C_i\neq\varnothing$ and
$f=\max_{i\in I} d_{C_i}$.
Let $x\in X\smallsetminus C$, set $I(x) = \menge{i\in I}{f(x)=d_{C_i}(x)}$,
and set $Q(x) = \conv\{P_{C_i}x\}_{i\in I(x)}$. 
Then 
\begin{equation}
\GG(x) = \bigcup_{q(x)\in Q(x)} \left\{x -
\frac{f^2(x)}{\|x-q(x)\|^2}\big(x-q(x)\big)\right\}
\quad\text{and}\quad
Q(x)\subseteq \conv\big(\{x\}\cup \GG(x)\big). 
\end{equation}
If $I(x)=\{i\}$ is a singleton, then $\GG(x)=\{P_{C_i}x\}$. 
\end{proposition}
\begin{proof}
This follows from Proposition~\ref{p:calc}\ref{p:calc6} and 
the fact that $\nabla d_{C_i}(x) = (x-P_{C_i}x)/d_{C_i}(x)$ when
$x\in X\smallsetminus C_i$. 
\end{proof}

\begin{proposition}
\label{p:hals}
Let $(C_i)_{i\in I}$ be a finite family of nonempty closed convex
subsets of $X$ such that $C = \bigcap_{i\in
I}C_i\neq\varnothing$.
Let $(\lambda_i)_{i\in I}$ be a family in $\left]0,1\right]$ such
that $\sum_{i\in I}\lambda_i = 1$.
Let $p\geq 1$ and suppose that
$f = \sum_{i\in I} \lambda_id_{C_i}^p$.
Set $(\forall x\in X)$
$I(x) = \menge{i\in I}{x\notin C_i}$.
Then 
\begin{equation}
(\forall x\in X)\quad
Gx = x - \frac{\sum_{i\in
I(x)}\lambda_id_{C_i}^p(x)}{p\big\|\sum_{i\in
I(x)}\lambda_id_{C_i}^{p-2}(x)(x-P_{C_i}x)\big\|^2}
\sum_{i\in I(x)}\lambda_i d_{C_i}^{p-2}(x)(x-P_{C_i}x)
\end{equation}
and if $p=2$, we rewrite this as 
\begin{equation}
(\forall x\in X)\quad
Gx = 
\begin{cases}
\displaystyle x - \frac{\sum_{i\in
I}\lambda_i\|x-P_{C_i}x\|^2}{2\big\|\sum_{i\in
I}\lambda_i(x-P_{C_i}x)\big\|^2}\Big(x-\sum_{i\in
I}\lambda_iP_{C_i}x\Big), &\text{if $x\notin C$;}\\
x, &\text{otherwise.}
\end{cases}
\end{equation}
\end{proposition}
\begin{proof}
Let $x\in X$, and let $i\in I$.
Then
$\nabla d_{C_i}(x) = d_C^{-1}(x)(x-P_{C_i}x)$ if $x\notin C_i$
and $0\in\partial d_{C_i}(x)$ otherwise.
Hence 
\begin{equation}
\nabla d_{C_i}^p(x) = 
pd_{C_i}^{p-2}(x)(x-P_{C_i}x)
\end{equation}
if $x\notin C_i$, and $0\in\partial  d_{C_i}^p(x)$ otherwise. 
The result follows.
\end{proof}

\begin{example}
\label{ex:dCp}
Let $p\geq 1$ and suppose that $f=d_C^p$.
Then $G = (1-\tfrac{1}{p})\Id + \tfrac{1}{p}P_C$.
\end{example}
\begin{proof}
This follows from Proposition~\ref{p:hals} when
$I$ is a singleton.
\end{proof}

\begin{example}
\label{ex:linear}
Suppose that $u\in X$ satisfies $\|u\|=1$,
and let $\beta\in\RR$.
Then the following hold:
\begin{enumerate}
\item
\label{ex:linear1}
If $f\colon x\mapsto \scal{u}{x}-\beta$,
then $C = \menge{x\in X}{\scal{u}{x}\leq\beta}$ and
$G\colon x\mapsto x- (\scal{u}{x}-\beta)^+u$.
\item
\label{ex:linear2}
If $f\colon x\mapsto |\scal{u}{x}-\beta|$,
then $C = \menge{x\in X}{\scal{u}{x}=\beta}$ and
$G\colon x\mapsto x- (\scal{u}{x}-\beta)u$.
\end{enumerate}
\end{example}
\begin{proof}
\ref{ex:linear1}:
Note that $f^+ = d_C$ and hence $G=P_C$ by 
Proposition~\ref{p:calc}\ref{p:calc7} and Example~\ref{ex:dCp}.
\ref{ex:linear2}:
Here $f=d_C$ and hence $G=P_C$ by Example~\ref{ex:dCp}.
\end{proof}

\begin{remark}
Using Example~\ref{ex:dCp}, we see that
$G$ is linear and that $G=G^*$ provided that
$f = d_C^p$, where $p\geq 1$ and $C$ is a subspace.
The converse is true as well but this lies beyond the scope of
this paper. 
\end{remark}

We now give two examples in which $G$ is positively homogenenous but 
not necessarily linear.

\begin{example}
Suppose that $f$ is a norm on $X$, with duality mapping
$J=\partial \tfrac{1}{2}f^2$. 
Then $C=\{0\}$ and
$(\forall x\in X\smallsetminus\{0\})$
$Gx = x - f^2(x)\|Jx\|^{-2}Jx$. 
\end{example}

\begin{example}
Let $K$ be a nonempty closed convex cone with polar
cone $K^\ominus$, and suppose that $f \colon
x\mapsto \thalb\scal{x}{P_Kx}$.
Then $G = \Id-\tfrac{1}{2}P_K = P_{K^\ominus}+\tfrac{1}{2}P_K$.
\end{example}
\begin{proof}
Since $(\forall x\in X)$
$f(x) = \thalb\|P_Kx\|^2 = \thalb d^2_{K^\ominus}(x)$,
it follows that $\nabla f(x) = x-P_{K^\ominus}x = P_Kx$.
The formula then follows.
\end{proof}

A direct verification yields the following result which
is well known when $p=2$. 

\begin{proposition}
Let $Y$ be another real Hilbert space,
let $A\colon X\to Y$ be continuous and linear,
let $b\in Y$, and let $\varepsilon\geq 0$,
and let $p\geq 1$. 
Suppose that $(\forall x\in X)$
$f(x) = \|Ax-b\|^p-\varepsilon^p$ and that
$C = \menge{x\in X}{\|Ax-b\|\leq\varepsilon}\neq\varnothing$.
Then
\begin{equation}
(\forall x\in X)\quad
Gx = \begin{cases}
\displaystyle x- \frac{\|Ax-b\|^p
-\varepsilon^p}{p\|Ax-b\|^{p-2}\|A^*(Ax-b)\|^2}A^*(Ax-b),
&\text{if $\|Ax-b\|>\varepsilon$;}\\
x, &\text{otherwise.}
\end{cases}
\end{equation}
\end{proposition}

\section{Continuity of $G$ vs \frechet\ differentiability of $f$} 

\label{s:contF}

We start with a technical result.

\begin{lemma}
\label{l:awards}
Let $(x_n)_\nnn$ be a sequence in $X$ converging
weakly to $\bar{x}$ and such that $x_n-Gx_n\to 0$. 
Suppose that one of the following holds:
\begin{enumerate}
\item 
\label{l:awards1}
$x_n\to\bar{x}$.
\item
$f$ is bounded on every bounded subset of $X$.
\end{enumerate}
Then $\bar{x}\in C$. 
\end{lemma}
\begin{proof}
Because of either \cite[Proposition~16.14]{BC2011} or \cite[Proposition~16.17]{BC2011}
there exists $\rho>0$ such that $\sigma := \sup\|\partial
f(\ball(\bar{x};\rho))\| < \pinf$. 
We thus can and do assume that 
\begin{equation}
(\forall\nnn)\quad \|s(x_n)\|\leq \sigma.
\end{equation}
Since $f^+$ is weakly lower semicontinuous, we deduce 
from Fact~\ref{f:known}\ref{f:known4} that
\begin{equation}
f^+(\bar{x}) \leq\varliminf f^+(x_n)
\leq \sigma\varliminf \|x_n-Gx_n\|= 0.
\end{equation}
Hence $f(\bar{x})\leq 0$, i.e., $\bar{x}\in C$. 
\end{proof}

\begin{remark}
Lemma~\ref{l:awards}\ref{l:awards1} and
Fact~\ref{f:known}\ref{f:known1} imply that $G$ is 
\emph{fixed-point closed} at $\bar{x}$
(see, e.g., also \cite[Theorem~4.2.7]{Cegielski} or \cite{BCW1}), i.e., 
if $x_n\to\bar{x}$ and $x_n-Gx_n\to 0$, then $\bar{x}=G\bar{x}$.
\end{remark}

\begin{proposition}
\label{p:couleen}
$G$ is continuous at every point in $C$.
\end{proposition}
\begin{proof}
Let $\bar{x}\in C$, and let $(x_n)_\nnn$ be a sequence in $X$
converging to $\bar{x}$. The result is clear if $(x_n)_\nnn$ lies in $C$,
so we can and do assume that $(x_n)_\nnn$ lies in $X\smallsetminus C$. 
Then $(\forall\nnn)$
$f(x_n) \leq f(\bar{x})-\scal{s(x_n)}{\bar{x}-x_n} \leq
\scal{s(x_n)}{x_n-\bar{x}} \leq \|s(x_n)\|\|\bar{x}-x_n\|$.
Hence 
$0 < f(x_n)/\|s(x_n)\|\leq \|\bar{x}-x_n\|\to 0$. 
By Fact~\ref{f:known}\ref{f:known4},
$x_n-Gx_n\to 0$. 
Thus $\lim Gx_n = \lim x_n = \bar{x}=G\bar{x}$ using
Fact~\ref{f:known}\ref{f:known1}. 
\end{proof}

The continuity of $G$ outside $C$ is more delicate.

\begin{fact}[Smulyan]
{\rm (See, e.g., \cite[Proposition~6.1.4]{BV}.)}
\label{f:nabla}
The following hold:
\begin{enumerate}
\item
\label{f:nablaF}
$f$ is \frechet\ differentiable at $\bar{x}$
$\Leftrightarrow$ 
$s$ is (strong-to-strong) continuous at $\bar{x}$.
\item
\label{f:nablaG}
$f$ is \gateaux\ differentiable at $\bar{x}$
$\Leftrightarrow$ 
$s$ is strong-to-weak continuous at $\bar{x}$.
\end{enumerate}
\end{fact}

\begin{lemma}
\label{l:131029b}
Suppose that $\bar{x}\in X\smallsetminus C$,
that $G$ is strong-to-weak continuous at $\bar{x}$,
but $G$ is not strong-to-strong continuous at $\bar{x}$.
Then $f$ is not \gateaux\ differentiable at $\bar{x}$. 
\end{lemma}
\begin{proof}
There exists a sequence 
$(x_n)_\nnn$ in $X\smallsetminus C$ such that
$x_n\to \bar{x}$, 
$Gx_n\weakly G\bar{x}$ yet 
$Gx_n\not\to G\bar{x}$.
It follows that 
\begin{equation}
x_n-Gx_n\weakly \bar{x}-G\bar{x}
\quad\text{and}\quad
x_n-Gx_n\not\to \bar{x}-G\bar{x}.
\end{equation}
By Kadec--Klee,
$\|x_n-G_n\| \not\to \|\bar{x}-G\bar{x}\|$.
Since $\|\cdot\|$ is weakly lower semicontinuous, 
we assume (after passing to a subsequence and relabeling if necessary) 
that
\begin{equation}
\|\bar{x}-G\bar{x}\| < \eta := \lim_\nnn \|x_n-Gx_n\|.
\end{equation}
Using Fact~\ref{f:known}\ref{f:known5}, it follows that
\begin{align}
s(x_n) &= f(x_n)\frac{x_n-Gx_n}{\|x_n-Gx_n\|^2}
\weakly f(\bar{x})\frac{\bar{x}-G\bar{x}}{\eta^2}
\neq f(\bar{x})\frac{\bar{x}-G\bar{x}}{\|\bar{x}-G\bar{x}\|^2} 
= s(\bar{x}).
\end{align}
Thus, $s$ is not strong-to-weak continuous at $\bar{x}$.
It follows now from Fact~\ref{f:nabla}\ref{f:nablaG} 
that $f$ is not \gateaux\ differentiable at $\bar{x}$. 
\end{proof}

\begin{theorem}
\label{c:131029c}
Let $\bar{x}\in X\smallsetminus C$. Then the following are equivalent:
\begin{enumerate}
\item
\label{c:131029c1}
$f$ is \frechet\ differentiable at $\bar{x}$.
\item
\label{c:131029c2}
$G$ is (strong-to-strong) continuous at $\bar{x}$.
\item
\label{c:131029c3}
$f$ is \gateaux\ differentiable at $\bar{x}$ and 
$G$ is strong-to-weak continuous at $\bar{x}$.
\end{enumerate}
\end{theorem}
\begin{proof}
``\ref{c:131029c1}$\Rightarrow$\ref{c:131029c2}'': 
By Fact~\ref{f:nabla}\ref{f:nablaF}, 
$s$ is continuous at $\bar{x}$. It follows from the definition of $G$
that $G$ is continuous at $\bar{x}$ as well. 

``\ref{c:131029c1}$\Leftarrow$\ref{c:131029c2}'': 
In view of Fact~\ref{f:known}\ref{f:known5}, 
we have
$s(x) = f(x)(x-Gx)/\|x-Gx\|^2$ for all $x$ sufficiently close to 
$\bar{x}$. 
Hence $s$ is continuous at $\bar{x}$ and therefore $f$ is \frechet\
differentiable at $\bar{x}$ by Fact~\ref{f:nabla}\ref{f:nablaF}. 

``\ref{c:131029c1}$\Rightarrow$\ref{c:131029c3}'' and 
``\ref{c:131029c2}$\Rightarrow$\ref{c:131029c3}'': 
This is clear since 
\ref{c:131029c1}$\Leftrightarrow$\ref{c:131029c2} by the above.

``\ref{c:131029c3}$\Rightarrow$\ref{c:131029c2}'': 
Suppose to the contrary that $G$ is not strong-to-strong continuous. 
Then, by Lemma~\ref{l:131029b}, $f$ is not \gateaux\ differentiable at
$\bar{x}$ which is absurd. 
\end{proof}

\begin{corollary}[continuity]
\label{c:awards}
$G$ is continuous everywhere if and only if
$f$ is \frechet\ differentiable on $X\smallsetminus C$. 
\end{corollary}
\begin{proof}
Combine Proposition~\ref{p:couleen} with 
Theorem~\ref{c:131029c}. 
\end{proof}

\begin{example}
Suppose that $X=\RR$ and that 
$(\forall x\in\RR)$
$f(x)=\max\{-x,x,2x-1\}$.
Then $C=\{0\}$ and $f$ is not 
differentiable at $1$;
consequently, by Corollary~\ref{c:awards},
$G$ is not continuous at $1$.
\end{example}

\begin{remark}[weak-to-weak continuity]
It is unrealistic to expect that $G$ is weak-to-weak continuous
even when $f$ is \frechet\ differentiable; 
see \cite[Example~3.2 and Remark~3.3.(ii)]{BCW1}. 
\end{remark}

\section{Continuity of $G$ vs \gateaux\ differentiability of $f$}

\label{s:contG}

In view of Fact~\ref{f:nabla} and Corollary~\ref{c:awards}, 
it is now tempting to conjecture that $G$ is strong-to-weak
continuous if and only if $f$ is \gateaux\ differentiable on
$X\smallsetminus C$. Perhaps somewhat surprisingly, this
turns out to be wrong. The counterexample is based on an
ingenious construction by Borwein and Fabian \cite{BF}.

\begin{example}[Borwein--Fabian]
\label{ex:BF}
{\rm (See \cite[Proof of Theorem~4]{BF}.)}
Suppose that $X$ is infinite-dimensional.
Then there exists a function $b\colon X\to\RR$ such that
the following hold:
\begin{enumerate}
\item 
\label{ex:BF1}
$b$ is continuous, convex and $\min b(X) = b(0) = 0$.
\item 
\label{ex:BF2}
$b$ is \frechet\ differentiable on $X\smallsetminus\{0\}$.
\item 
\label{ex:BF3}
$b$ is \gateaux\ differentiable at $0$, and $\nabla b(0)=0$.
\item 
\label{ex:BF4}
$b$ is not \frechet\ differentiable at $0$.
\end{enumerate}
\end{example}

\begin{example}[lack of strong-to-weak continuity]
\label{ex:131029f}
Let $b$ be as in Example~\ref{ex:BF}. 
Then there exists $y\in X$ such that $\nabla b(y)\neq 0$.
Suppose that
\begin{equation}
(\forall x\in X)\quad
f(x) = b(x) - \scal{\nabla b(y)}{x} 
-\tfrac{1}{2}\big(b(y)-\scal{\nabla b(y)}{y}\big).
\end{equation}
Then the following hold:
\begin{enumerate}
\item
\label{ex:131029f5}
$f$ is \gateaux\ differentiable (but not \frechet\
differentiable) at $0$, and
$G$ is not strong-to-weak continuous at $0$. 
\item
\label{ex:131029f6}
$f$ is \frechet\ differentiable on $X\smallsetminus\{0\}$, and 
$G$ is continuous on $X\smallsetminus\{0\}$. 
\end{enumerate}
\end{example}
\begin{proof}
By Example~\ref{ex:BF}\ref{ex:BF3}, $0\in\ran \nabla b$.
If $\{0\} = \ran \nabla b$, then we would deduce 
that $b$ is constant
and therefore \frechet\ differentiable; in turn, this 
would contradict Example~\ref{ex:BF}\ref{ex:BF4}. 
Hence $\{0\}\subsetneqq \ran\nabla b$ and there exists
$y\in X$ such that 
\begin{equation}
\label{e:0308a}
v = \nabla b(y)\neq 0.
\end{equation}
Now set 
\begin{equation}
g\colon X\to\RR\colon x\mapsto b(x)-\scal{v}{x}.
\end{equation}
Then 
\begin{equation}
(\forall x\in X)\quad
f(x) = g(x)-\tfrac{1}{2}g(y), 
\end{equation}
and $g(0)=b(0)-\scal{v}{0}=0$ by
Example~\ref{ex:BF}\ref{ex:BF1}.
Example~\ref{ex:BF}\ref{ex:BF3} and \eqref{e:0308a} yield
$\nabla g(0)=\nabla b(0)-v=-v\neq 0$
while $\nabla g(y)=\nabla b(y)-v=0$.
Hence
$\min g(X) = g(y)<g(0) = 0$ and therefore
\begin{equation}
f(y) = \min f(X) = \min g(X) - \tfrac{1}{2}g(y) = \tfrac{1}{2}g(y)< 0 <
0 - \tfrac{1}{2}g(y) = f(0).
\end{equation}
Thus $y\in C$ while $0\notin C$. 

\ref{ex:131029f5}: 
On the one hand, since $b$ is not \frechet\ differentiable at $0$
(Example~\ref{ex:BF}\ref{ex:BF4}), neither is $f$. 
On the other hand, since $b$ is \gateaux\ differentiable at $0$
(Example~\ref{ex:BF}\ref{ex:BF3}),
so is $f$. 
Altogether, $f$ is \gateaux\ differentiable, but not \frechet\
differentiable, at $0$. 
Therefore, by Theorem~\ref{c:131029c}, $G$ is not strong-to-weak
continuous at $0$. 

\ref{ex:131029f6}: 
Since $b$ is \frechet\ differentiable on $X\smallsetminus\{0\}$
(Example~\ref{ex:BF}\ref{ex:BF2}), so is $f$. 
Now apply Theorem~\ref{c:131029c}. 
\end{proof}

\section{$G$ as an ``accelerated mapping''}

In this section, we consider the case when $f$ is a power of a
quadratic form. 

\label{s:Frank}

\begin{proposition}
\label{p:0308b}
Suppose that 
$f\colon x\mapsto \sqrt{\scal{x}{Mx}^p}$, where
$p\geq 1$ and
$M\colon X\to X$ be continuous, linear, self-adjoint, and
positive. 
Then $G$ is continuous everywhere and 
\begin{equation}
(\forall x\in X)\quad
Gx = 
\begin{cases}
\displaystyle x - \frac{\scal{x}{Mx}}{p\|Mx\|^2}Mx, &\text{if $Mx\neq 0$;}\\
x, &\text{if $Mx=0$.}
\end{cases}
\end{equation}
\end{proposition}
\begin{proof}
Assume first that $p=1$.
Since $M$ has a unique positive square root, i.e.,
there exists\footnote{See, e.g., \cite[Theorem~9.4-2]{Kreyszig},
where this is stated in a complex Hilbert space; however, the proof
works unchanged in our real setting as well.} 
$B\colon X\to X$ such that $B$ is continuous, linear,
self-adjoint, and positive, and $\ker B = \ker M$.
Hence $(\forall x\in X)$ $f(x)= \sqrt{\scal{x}{Mx}} = \|Bx\|$
so $f$ is indeed convex and continuous.
If $x\in X\smallsetminus\ker M = X\smallsetminus
\ker B$, then $f$ is \frechet\ differentiable at $x$ with 
$\nabla f(x) = B^*Bx/\|Bx\| = Mx/\|Bx\|$; hence,
\begin{equation}
Gx = 
x - \frac{\|Bx\|}{\|Mx\|^2/\|Bx\|^2}\frac{Mx}{\|Bx\|} = 
x - \frac{\|Bx\|^2}{\|Mx\|^2}Mx =
x - \frac{\scal{x}{Mx}}{\|Mx\|^2}Mx
\end{equation}
and $G$ is continuous everywhere by Corollary~\ref{c:awards}. 
If $p>1$, then the result follows from the above and
Proposition~\ref{p:calc}\ref{p:calc3}. 
\end{proof}

\begin{example}
\label{ex:Frank}
Let $A\colon X\to X$ be linear, self-adjoint, and nonexpansive. 
Suppose that $(\forall x\in X)$
$f(x) = \sqrt{\scal{x}{x-Ax}}$.
Then $G$ is continuous everywhere and 
\begin{equation}
(\forall x\in X)\quad 
Gx= 
\begin{cases}
\displaystyle x - \frac{\scal{x}{x-Ax}}{\|x-Ax\|^2}(x-Ax), &\text{if
$Ax\neq x$;}\\
x, &\text{if $Ax=x$.}
\end{cases}
\end{equation}
\end{example}
\begin{proof}
Use Proposition~\ref{p:0308b} with $M=\Id-A$ and $p=1$. 
\end{proof}

\begin{remark}[accelerated mapping]
Let $A\colon X\to X$ be linear, nonexpansive, and self-adjoint. 
In \cite{BDHP}, the authors study the accelerated
mapping\footnote{In fact, the operator $A$ in \cite{BDHP} need 
not necessarily be self-adjoint.} of $A$,
i.e., 
\begin{equation}
x \mapsto t_xAx+ (1-t_x)x,
\quad
\text{where }
t_x = \begin{cases}
\displaystyle \frac{\scal{x}{x-Ax}}{\|x-Ax\|^2}, &\text{if $x\neq
Ax$;}\\
1, &\text{otherwise.}
\end{cases}
\end{equation}
In view of the Example~\ref{ex:Frank}, the accelerated mapping of
$A$ is precisely the subgradient projector $G$ of the function
$x\mapsto \sqrt{\scal{x}{x-Ax}}$.
Now suppose that $X = \ell^2(\NN)$, let $(e_n)_\nnn$ be the
standard orthonormal basis of $X$, and suppose that
\begin{equation}
A\colon X\to X\colon x\mapsto \sum_\nnn
\tfrac{n}{n+1}\scal{e_n}{x}e_n.
\end{equation}
Then $G$ is continuous (Example~\ref{ex:Frank}); however,
$G$ is neither linear nor uniformly continuous 
(see the \cite[Remark following Lemma~3.8]{BDHP}). 
\end{remark}

\section{Nonexpansiveness}

We now discuss when $G$ is (firmly) nonexpansive or monotone.

\label{s:nonexp}

\begin{proposition}
\label{p:nonexp}
Suppose that $f$ is \gateaux\ differentiable on $X\smallsetminus
C$ and that $G_f$ is firmly nonexpansive.
Then $G_g$ is likewise in each of the following situations:
\begin{enumerate}
\item
\label{p:nonexp1}
$\alpha >0$, and $g=f\circ \alpha\Id$ is convex.
\item
\label{p:nonexp2}
$f\geq 0$, $\alpha\geq 1$, and $g=f^\alpha$ is convex.
\item
$A\colon X\to X$ is continuous and linear, $AA^*=A^*A=\Id$,
and $g=f\circ A$.
\item
$z\in X$ and $g\colon x\mapsto f(x-z)$.
\end{enumerate}
The analogous statement holds when $G_f$ is assumed to be
nonexpansive. 
\end{proposition}
\begin{proof}
This follows from the corresponding items in
Proposition~\ref{p:calc}, which do preserve (firm)
nonexpansiveness.
\end{proof}

On the real line, we obtain a simpler test.

\begin{proposition}
\label{p:changeclock}
Suppose that $X=\RR$ and that
$f$ is twice differentiable on $X\smallsetminus C$.
Then $G$ is monotone.
Moreover, $G$ is (firmly) nonexpansive if and only if
\begin{equation}
(\forall x\in\RR)\quad
f(x)f''(x) \leq \big(f'(x)\big)^2.
\end{equation}
\end{proposition}
\begin{proof}
By Corollary~\ref{c:awards}, $G$ is continuous. 
Let $x\in\RR\smallsetminus C$.
Then $G(x) = x-f(x)/f'(x)$ and hence
$G'(x) = f(x)f''(x)/(f'(x))^2\geq 0$.
It follows that $G$ is increasing on $X\smallsetminus C$ and
hence on $\RR$. 
Furthermore, $G$ is (firmly) nonexpansive if and only if
$G'(x)\leq 1$, which gives the remaining characterization.
\end{proof}

\begin{example}
Suppose that $X=\RR$,
let $\alpha>0$, and suppose that $(\forall x\in\RR)$
$f(x)=x^n-\alpha$, where $n\in\{2,4,6,8,\ldots\}$. 
Then $G$ is firmly nonexpansive.
\end{example}
\begin{proof}
If $x\in\RR\smallsetminus C$, then
$(f'(x))^2 -f(x)f''(x)= nx^{n-2}(\alpha n + x^n-\alpha)>0$ and 
we are done by 
Proposition~\ref{p:changeclock}.
\end{proof}


\begin{example}
Suppose that $X=\RR$ and that $f\colon x\mapsto\exp(|x|)-1$.
Then $(\forall x\in X)$ $G(x) = x-\sgn(x)(1-\exp(-|x|))$ and 
$G'(x)=1-\exp(-|x|)\in \left[0,1\right[$.
It follows that $G$ is firmly nonexpansive\footnote{Since $G$ is monotone
by Proposition~\ref{p:changeclock}, its antiderivative
$x\mapsto \tfrac{1}{2}x^2 - |x|-\exp(-|x|)$ is convex --- although this does not
look like convex function on first glance! It is interesting to do this
also for other instances of $f$.}.
\end{example}

\begin{example}
Suppose that $X=\RR$ and that
$f \colon x\mapsto \exp(x^2)-1$.
Then $G$ is not (firmly) nonexpansive.
Indeed, we compute 
$(f'(x))^2-f(x)f''(x)= 4x^2\exp(x^2)+2\exp(x^2)-2\exp(2x^2)$,
which strictly negative when $|x|>1.2$. 
Now apply Proposition~\ref{p:changeclock}. 
\end{example}

\begin{proposition}
Suppose that $X=\RR$ and that 
$f$ is twice differentiable, that $\min f(X)=0$, that
$g = f\Box (1/2)|\cdot|^2$, 
and that $2ff'' \leq (2+f'')(f')^2$. 
Then $G_g$ is firmly nonexpansive.
\end{proposition}
\begin{proof}
We start by observing a couple of facts.
First, 
\begin{equation}
g' = \Id-\prox_f.
\end{equation}
Write $y=\prox_f(x)$. Then
$x=y+f'(y)$ and hence implicit differentiation gives
$1= y'(x)+f''(y)y'(x)=y'(x)(1+f''(y(x)))$.
Hence $y'=1/(1+f''(y(x)))$ and thus
\begin{equation}
g''(x)=\big(\Id-\prox_f\big)'(x) = 1- \frac{1}{1+f''(\prox_f(x))}
= \frac{f''\big(\prox_f(x)\big)}{1+f''\big(\prox_f(x)\big)}.
\end{equation}
In view of Proposition~\ref{p:changeclock} 
and because $g(x)=f(\prox_f(x)) + (1/2)(x-\prox_f(x))^2$
we must verify that 
$gg''\leq (g')^2$, i.e.,
\begin{equation}
\label{e:0317a}
\frac{\big(f(\prox_f(x)) + \thalb(x-\prox_f(x))^2   \big)f''\big(\prox_f(x)\big)}{1+f''\big(\prox_f(x)\big)}
\leq \big(x-\prox_f(x)\big)^2.
\end{equation}
Again writing $y=\prox_f(x)$ gives $x-\prox_f(x)=f'(y)$ and so 
see that \eqref{e:0317a} is equivalent to 
\begin{equation}
\label{e:0317b}
\frac{\big(f(y) + \thalb(f'(y))^2 \big)f''(y)}{1+f''(y)}
\leq \big(f'(y)\big)^2.
\end{equation}
However, \eqref{e:0317b} holds by our assumption on $f$. 
\end{proof}

We conclude this section with a result on the range of $\Id-G$.

\begin{proposition}
We have $\ran(\Id-G)\subseteq\cone\ran\partial f \subseteq(\reck
C)^\ominus$. 
\end{proposition}
\begin{proof}
Let $y^*\in\partial f(y)$,
let $c\in C$, and let $x\in\reck C$.
Then $(c+nx)_\nnn$ lies in $C$.
Hence $(\forall n\geq 1)$ 
$0\geq f(c+nx)\geq f(y)+\scal{y^*}{c+nx-y}$ and thus
\begin{equation}
\scal{y^*}{x} \leq \frac{\scal{y^*}{y-c}-f(y)}{n}\to 0
\quad\text{as $n\to\pinf$.}
\end{equation}
It follows that $y^*\in(\reck C)^\ominus$.
Therefore, $\ran(\Id-G)\subseteq \cone\ran\partial f\subseteq
(\reck C)^\ominus$.
\end{proof}

\section{The decreasing property}

\label{s:decrease}

We say that $f$ has the \emph{decreasing property}
if 
\begin{equation}
(\forall x\in X) \quad \sup f(\GG x)\leq f(x).
\end{equation}
To verify this, it suffices to consider points outside $C$.

\begin{proposition}
\label{p:plumber}
If $(\forall x\in X)$ $Gx \in \conv(\{x\}\cup C)$,
then $f$ has the decreasing property.
\end{proposition}
\begin{proof}
Let $x\in X\smallsetminus C$. 
Then there exists $c\in C$ and $\lambda\in[0,1]$ such that
$Gx = (1-\lambda)x+\lambda c$.
It follows that
$f(Gx) \leq (1-\lambda)f(x)+\lambda f(c) 
\leq (1-\lambda) f(x)\leq f(x)$.
\end{proof}

\begin{lemma}
\label{l:plumber}
Let $(x,y,z)\in\RR^3$ be such that $x\neq z$
and $(z-y)(x-y)\leq 0$.
Then $y\in\conv\{x,z\}$.
\end{lemma}
\begin{proof}
Suppose first that $z<x$.
If $y>x$, then $(z-y)(x-y)>0$ because it is the product of two
strictly negative numbers.
Similarly, if $y<z$, then $(z-y)(x-y)>0$.
We deduce that $y\in[z,x]$.
Analogously, when $x<z$, we obtain that $y\in[x,z]$.
In either case, $y\in\conv\{x,z\}$.
\end{proof}

\begin{corollary}
\label{c:plumber}
Suppose that $X=\RR$.
Then $f$ has the decreasing property. 
\end{corollary}
\begin{proof}
Let $x\in\RR\smallsetminus C$.
Then $x\neq P_Cx$ and, 
by Fact~\ref{f:known}\ref{f:known3},
$(P_Cx-Gx)(x-Gx)\leq 0$. 
Lemma~\ref{l:plumber} thus yields
$Gx\in\conv\{x,P_Cx\}$.
Hence $Gx\in \conv(\{x\}\cup C)$, and
we are done by Proposition~\ref{p:plumber}. 
\end{proof}

The next example shows that the decreasing property is not automatic.

\begin{example}
Suppose that $X=\RR^2$,
that $C_1 =\RR\times\{0\}$, that 
$C_2 = \menge{(\xi,\xi)\in X}{\xi\in\RR}$,
and that $f=\max\{d_{C_1},d_{C_2}\}$.
Then $f$ does not have the decreasing property.
\end{example}
\begin{proof}
Set $x=(2,1)$. Then, using Proposition~\ref{p:maxdist}, we obtain
that $Gx=(2,0)$ and $f(x)=1<\sqrt{2}=f(Gx)$.
\end{proof}

We now illustrate that the sufficient condition of
Proposition~\ref{p:plumber} is not necessary:

\begin{example}
\label{ex:ell1}
Suppose that $X=\RR^2$ and that
$(\forall x=(x_1,x_2)\in\RR^2)$ 
$f(x)=|x_1|+|x_2|$.
Then $f$ has the decreasing property, $G^2x=(0,0)$ 
yet $Gx\notin \conv\{(0,0),x\}$ for almost every $x\in\RR^2$.
Furthermore, $G$ is not monotone.
\end{example}
\begin{proof}
Observe that $C=\{(0,0)\}$. 
Let $I=\{1,2,3,4\}$ and 
consider the four halfspaces $(C_i)_{i\in I}$
with normal vectors
$(1,1)$ and $(1,-1)$ with $(0,0)$ in their boundaries,
and with the two boundary hyperplanes $H_1$ and $H_2$. 
Then $f= \sqrt{2}\max_{i\in I} d_{C_i} = \sqrt{2}\max\{d_{H_1},d_{H_2}\}$ 
by Example~\ref{ex:linear}\ref{ex:linear1}.
Proposition~\ref{p:maxdist} implies that 
$G$ is the projector onto the farther hyperplane on 
$\RR^2\smallsetminus S$,
where $S = (\RR\times\{0\})\cup(\{0\}\times\RR)$. 
It is thus clear that
$Gx\notin \conv\{(0,0),x\}$ and that
$f(Gx)\leq f(x)$ for every $x\in\RR^2\smallsetminus S$.
When $x\in S$, one checks directly that $f(Gx)\leq f(x)$. 
Hence $f$ has the decreasing property. 
Finally, let $x=(-1,3)$ and $y=(1,3)$.
Then $Gx=(1,1)$ and $Gy=(-1,1)$ and hence
$\scal{x-y}{Gx-Gy}=-4<0$ so $G$ is not monotone.
\end{proof}

\begin{remark}[infeasibility detection]
Using the decreasing property, one obtains a sufficient condition
for \emph{infeasibility}:
Suppose that $X=\RR$ and we find a point $x$ such
that $f(Gx)>f(x)$.
Then $C$ must be empty because of Corollary~\ref{c:plumber}.
For instance, suppose that $f\colon x\mapsto x^2+1$.
Then 
\begin{equation}
\label{e:chaotic}
(\forall x\in\RR\smallsetminus\{0\})\quad
Gx = (x^2-1)/(2x).
\end{equation}
Now set $x = 1/2$.
Then $Gx=-3/4$ and
$f(Gx)=25/16>5/4=f(x)$. 
\end{remark}

\begin{remark}[Newton iteration]
Suppose that $X=\RR$ and that $f$ is differentiable on
$X\smallsetminus C$. Then 
\begin{equation}
(\forall x\in\RR\smallsetminus C)\quad
Gx = x - \frac{f(x)}{\big(f'(x)\big)^2}f'(x) = x - \frac{f(x)}{f'(x)}
\end{equation}
is the same as the Newton operator for  finding a zero of $f$!
It is known since the 19th century
that the concrete instance \eqref{e:chaotic}
exhibits chaotic behaviour; see, e.g., \cite[Problem~7-a on
page~72]{Milnor}. 
\end{remark}

The decreasing property is preserved in certain cases:

\begin{proposition}
\label{p:decalc}
Suppose that $f$ has the decreasing property.
Then the following hold:
\begin{enumerate}
\item
\label{p:decalc1}
If $\alpha>0$, then $\alpha f$ has the decreasing property.
\item
\label{p:decalc2}
If $\alpha\geq 1$, then $(f^+)^\alpha$ has the decreasing property.
\end{enumerate}
\end{proposition}
\begin{proof}
Let $x\in X\smallsetminus C$. 
\ref{p:decalc1}:
Then $(\alpha f)\GG_{\alpha f}(x)=(\alpha f)\GG_f(x)
\leq \alpha f(x)=(\alpha f)(x)$ by Proposition~\ref{p:calc}\ref{p:calc1}.
\ref{p:decalc2}:
Set $g=(f^+)^\alpha$ and $\beta=1/\alpha$.
Then $0<\beta\leq 1$ and 
$\GG_g(x) = (1-\beta)x + \beta\GG_f(x)$ by
Proposition~\ref{p:calc}\ref{p:calc3}. 
Hence 
$\sup g(\GG_gx) \leq (1-\beta)g(x)+\beta\sup g(\GG_fx)$. 
On the other hand,
$\sup g(\GG_f(x)) \leq g(x)$ by definition of $g$.
Altogether, $\sup g (\GG_gx)\leq g(x)$, i.e., $g$ is decreasing. 
\end{proof}

The following result is complementary to the decreasing property.

\begin{proposition}
\label{p:striconv}
Suppose that $f$ is strictly convex at $x\in X$ and 
$f(x)>0$.
Then $f(Gx)>0$. 
\end{proposition}
\begin{proof}
Recall that $f$ is strictly convex at $x$ if
$(\forall y\in X\smallsetminus\{x\})$
$(\forall\lambda\in\zeroun)$ 
$f((1-\lambda)x+\lambda y)<(1-\lambda)f(x)+\lambda f(y)$.
Arguing as in \cite[proof of Proposition~5.3.4.(a)]{BV}, we see
that 
$\thalb\scal{s(x)}{Gx-x} = \scal{s(x)}{(\thalb x+\thalb Gx)-x}
\leq f(\thalb x + \thalb Gx)-f(x)<
\thalb f(x)+\thalb f(Gx)-f(x) = \thalb(f(Gx)-f(x))$.
Therefore, 
$f(Gx)>f(x)+\scal{s(x)}{Gx-x} = 0$ using
Fact~\ref{f:known}\ref{f:known0}. 
\end{proof}

\begin{remark}
\label{r:striconv}
Suppose that $f$ is strictly convex.
Then Proposition~\ref{p:striconv} shows that
iterating $G$ starting at a point outside $C$ will never reach
$C$ in finitely many steps. 
This is clearly illustrated by Example~\ref{ex:dCp},
which shows that the function $d_C$, even though it is neither strictly convex
nor differentiable everywhere, performs best because $G=P_C$ yields
a solution after just one step. 
\end{remark}

\section{The subgradient projector of $(x_1,x_2)\mapsto
|x_1|^p+|x_2|^p$}

\label{s:pnorm}

The following result complements
Example~\ref{ex:ell1}.

\begin{proposition}
\label{p:ellp}
Suppose that $X=\RR^2$
and that $f\colon (x_1,x_2)\mapsto |x_1|^p + |x_2|^p$,
where $p>1$, and let $x=(x_1,x_2)\in\RR^2\smallsetminus\{(0,0)\}$.
Then 
\begin{equation}
\label{e:ellpG}
Gx = \left( x_1 -
\frac{\big(|x_1|^p+|x_2|^p\big)|x_1|^{p-1}\sgn(x_1)}{p\big(|x_1|^{2p-2}+|x_2|^{2p-2}\big)},
x_2 -
\frac{\big(|x_1|^p+|x_2|^p\big)|x_2|^{p-1}\sgn(x_2)}{p\big(|x_1|^{2p-2}+|x_2|^{2p-2}\big)}\right)
\end{equation}
and the following hold:
\begin{enumerate}
\item
\label{p:ellp1}
If $p\geq 2$, then
$f(x)\geq f(Gx)\geq (1-2p^{-1})^pf(x)$.
\item
\label{p:ellp2}
If $1<p\leq 2$, then
$f(x)\geq f(Gx)\geq 2^{-1}(1-p^{-1})^pf(x)$.
\item
\label{p:ellp3}
If $1<p<2$, then
$G$ is not monotone.
\end{enumerate}
\end{proposition}
\begin{proof}
The formula \eqref{e:ellpG} is a direct verification,
and \ref{p:ellp1}\&\ref{p:ellp2} hold when $x_1=0$ or $x_2=0$. 
We thus assume that $x_1\neq 0$ and $x_2\neq 0$. 

\ref{p:ellp1}:
Note that
\begin{equation}
f(Gx)=|x_1|^p\big|1-c_1\big|^p + |x_2|^p\big|1-c_2\big|^p,
\quad\text{where}\quad
c_i =
\frac{\big(|x_1|^p+|x_2|^p\big)|x_i|^{p-2}}{p\big(|x_1|^{2p-2}+|x_2|^{2p-2}\big)}.
\end{equation}
If $i\in\{1,2\}$ and $m\in\{1,2\}$ is such that $|x_m|=\max\{|x_1|,|x_2|\}$,
then 
$c_i \leq (2|x_m|^p|x_m|^{p-2})p^{-1}(|x_m|^p+0)^{-1}= 2/p$.
Hence $1\geq 1-c_i\geq 1-2p^{-1}\geq 0$ and the inequalities follow.

\ref{p:ellp2}:
We assume that $|x_1|\leq|x_2|$, the other case is treated analogously. 
Set $t = |x_1/x_2|$, 
\begin{equation}
c_1 = t - \frac{t^{2p-1}+t^{p-1}}{p\big(1+t^{2p-2}\big)}
\quad\text{and}\quad
c_2 = 1- \frac{1+t^p}{p\big(1+t^{2p-2}\big)},
\end{equation}
and check that
\begin{equation}
\label{e:piday4}
f(Gx) = |x_2|^p\big(|c_1|^p + |c_2|^p\big). 
\end{equation}
Since $p-2\leq 0$, we have $t^{p-2}\geq 1$ and hence
\begin{equation}
\label{e:piday5}
1 \geq c_2 \geq 1 - \frac{1+t^p}{p\big(1+t^p\big)} = 1- \tfrac{1}{p} \geq 0.
\end{equation}
Thus $c_2\geq 0$.
We now claim that
\begin{equation}
\label{e:piday1}
|c_1|+c_2 \leq 1.
\end{equation}
This will imply $\max\{|c_1|,|c_2|\}\leq 1$;
hence 
$\max\{|c_1|^p,|c_2|^p\}\leq 1$, 
\begin{equation}
f(Gx) \leq |x_2|^p\big(|c_1|+|c_2|\big)\leq |x_2|^p\leq f(x),
\end{equation}
and the decreasing property of $f$ follows.
Observe that \eqref{e:piday1} is equivalent to 
\begin{subequations}
\label{e:piday2}
\begin{align}
c_1+c_2 &\leq 1 \label{e:piday2a}\\
-c_1+c_2&\leq 1 \label{e:piday2b}
\end{align}
\end{subequations}
and hence to 
\begin{subequations}
\label{e:piday3}
\begin{align}
t &\leq \frac{(1+t^p)(1+t^{p-1})}{p(1+t^{2p-2})}\label{e:piday3a}\\
\frac{t^{p-1}(1+t^p)}{p(1+t^{2p-2})} &\leq t+
\frac{1+t^p}{p(1+t^{2p-2})}. \label{e:piday3b}
\end{align}
\end{subequations}
Now check that \eqref{e:piday3} holds by 
using $t^{p-1}\leq 1$ and, for \eqref{e:piday3a}, the 
convexity of $h\colon \xi\mapsto 1+\xi^p$,
which implies $h(t)\geq h(1)+h'(1)(t-1)$, i.e.,
$pt \leq 1+t^p$.
Furthermore, using \eqref{e:piday4}, \eqref{e:piday5} and
the assumption that $|x_2|\geq|x_1|$, 
we obtain 
\begin{equation}
f(Gx) \geq c_2^p|x_2|^p \geq \big(1-\tfrac{1}{p}\big)^p|x_2|^p
\geq \big(1-\tfrac{1}{p}\big)^p \frac{|x_1|^p+|x_2|^p}{2} = 
\frac{\big(1-\tfrac{1}{p}\big)^p}{2} f(x).
\end{equation}
\ref{p:ellp3}:
Consider the points $y=(1,\xi)$ and $z=(-1,\xi)$, where $\xi>0$.
Then $y-z=(2,0)$ and
\begin{subequations}
\begin{equation}
Gy = \left( 1 - \frac{1+\xi^p}{p(1+\xi^{2p-2})}, \xi - 
\frac{(1+\xi^p)\xi^{p-1}}{p(1+\xi^{2p-2})}\right)
\end{equation}
and 
\begin{equation}
Gz = \left( -1 + \frac{1+\xi^p}{p(1+\xi^{2p-2})}, \xi - 
\frac{(1+\xi^p)\xi^{p-1}}{p(1+\xi^{2p-2})}\right). 
\end{equation}
\end{subequations}
It follows that
\begin{equation}
\scal{Gy-Gz}{y-z} = 
4 \left( 1 - \frac{1+\xi^p}{p(1+\xi^{2p-2})}\right)<0
\quad
\text{as $\xi\to\pinf$}
\end{equation}
because 
$\lim_{\xi\to\pinf} (1+\xi^p)p^{-1}/(1+\xi^{2p-2})
= \lim_{\xi\to\pinf} (2p-2)^{-1}\xi^{2-p}=\pinf$ using
l'H\^{o}pital's rule. 
Therefore, $G$ is not monotone. 
\end{proof}

\begin{remark}
The operator $G$ of Proposition~\ref{p:ellp} seems to defy an
easy analysis. It would be interesting to obtain
complete characterizations in terms of $p$ of the 
following, increasingly more restrictive, properties:
$G$ is monotone;
$\Id-G$ is nonexpansive;
$G$ is firmly nonexpansive. 
With the help of Maple it is possible to check the following
statements:
\begin{enumerate}
\item
If $p\in\{2,4,6\}$, then
$G$ is firmly nonexpansive and hence monotone.
\item
If $p\in\{8,10,12\}$, then
$G$ is not firmly nonexpansive; however,
$\Id-G$ is nonexpansive and $G$ is monotone\,\footnote{Experiments
with Maple suggest that this pattern may hold true for every even
integer greater than or equal $8$.}.
\end{enumerate}
Suppose first that $p\in\{2,4,6\}$.
Then $G$ is firmly nonexpansive
$\Leftrightarrow$
$N=2G-\Id$ is nonexpansive
$\Leftrightarrow$
$(\forall x\in X)$ $Jx$ is nonexpansive, where $Jx$ is the Jacobian of $N$
at $x$
$\Leftrightarrow$
$(Jx)^*Jx\preceq \Id$
$\Leftrightarrow$
$\Id-(Jx)^*Jx$ is positive semidefinite. 
The last condition leads to checking three inequalities
using the principal minor criterion for positive semidefiniteness.
Dividing by appropriate powers of $x_1$ and $x_2$, 
this reduces to checking whether three polynomials in one
variable are
positive. Sturm's Theorem
(see, e.g., \cite[Theorem~1.4.3]{Prasolov}), which is implemented in
Maple and Mathematica, combinded with
\cite[Theorem~1.1.2]{Prasolov} finally complete the verification. 

Now suppose that $p\in\{8,10,12\}$. 
The approach just outlined shows that
$G$ is not firmly nonexpansive. 
Note the implications:
$G$ is monotone
$\Leftarrow$
$N=\Id-G$ is nonexpansive
$\Leftrightarrow$
$(\forall x\in X)$ $Jx$ is nonexpansive, 
where $Jx$ is the Jacobian of $N$ at $x$
$\Leftrightarrow$
$(Jx)^*Jx\preceq \Id$
$\Leftrightarrow$
$\Id-(Jx)^*Jx$ is positive semidefinite,
which is checked using again Sturm's Theorem. 
\end{remark}

\section{$G$ and the Yamagishi--Yamada operator}

\label{s:YY}

In this last section we study the accelerated version of $G$
proposed by Yamagishi and Yamada in \cite{YY}.
For fixed $L>0$ and $r>0$, 
we assume in addition that 
\begin{empheq}[box=\mybluebox]{equation}
\text{$f$ is \frechet\ differentiable and $\nabla f$
is Lipschitz continuous with constant $L$},
\end{empheq}
and that 
\begin{empheq}[box=\mybluebox]{equation}
\text{$f$ is bounded below with $\inf f(X)\geq -\rho$,}
\end{empheq}
and we set
\begin{empheq}[box=\mybluebox]{equation}
(\forall x\in X)\quad
\theta(x) = \frac{\|\nabla f(x)\|^2}{2L}-\rho.
\end{empheq}
By \cite[Lemma~1]{YY}, we have
\begin{equation}
\label{e:ftheta}
f \geq \theta.
\end{equation}
The Yamagishi--Yamada operator \cite{YY} is
\begin{empheq}[box=\mybluebox]{equation}
Z\colon X\to X,
\end{empheq}
defined at $x\in X$ by 
\begin{empheq}[box=\mybluebox]{equation}
\label{e:Z}
Zx= 
\begin{cases}
x, &\text{if $f(x)\leq 0$;}\\[+5mm]
\displaystyle 
x - \frac{\nabla f(x)}{\|\nabla f(x)\|^2} \, f(x),
&\text{if $f(x)>0$ and $\theta(x)\leq 0$;}\\[+5mm]
\displaystyle x - \frac{\nabla f(x)}{\|\nabla f(x)\|^2}
\,\Big(f(x)+\big(\textstyle \sqrt{\theta(x)+\rho}-\sqrt{\rho}\big)^2\Big),
&\text{if $f(x)>0$ and $\theta(x)>0$.}
\end{cases}
\end{empheq}
Note that if $f(x)\leq 0$ or $\theta(x)\leq 0$, then
$Zx=Gx$.

We now prove that if $X=\RR$, then $Z$ is itself a subgradient
projector. 

\begin{theorem}
\label{t:caifang}
Suppose that $X=\RR$ and that $f$ is also twice differentiable.
Then for every $x\in \RR$, 
\eqref{e:Z} can be rewritten as 
\begin{equation}
\label{e:Z1}
Zx= 
\begin{cases}
x, &\text{if $f(x)\leq 0$;}\\[+5mm]
\displaystyle 
x - \frac{1}{f'(x)}\,f(x), 
&\text{if $f(x)>0$ and $|f'(x)|\leq\sqrt{2L\rho}$;}\\[+5mm]
\displaystyle x - \frac{1}{f'(x)}
\,\left(f(x)+\left(\frac{|f'(x)|}{\sqrt{2L}}-\sqrt{\rho}\right)^2\right),
&\text{if $f(x)>0$ and $|f'(x)|>\sqrt{2L\rho}$.}
\end{cases}
\end{equation}
Set $D=\menge{x\in X}{\theta(x)\leq 0}$ 
and assume
that $\bd D \subseteq X\smallsetminus C$.
Then $D$ is a closed convex superset of $C$, and $Z$ is a subgradient projector of a function $y$,
defined as follows.
On $D$, we set $y$ equal to $f$.
The set $\RR\smallsetminus D$ is empty, or an open interval,
or the disjoint union of two open intervals.
Assume that $I$ is one of these nonempty intervals, and
let $q$ be defined on $I$ such that
\begin{equation}
(\forall x\in I)\quad
q'(x) = \frac{1}{x-Zx}.
\end{equation}
Now set $d=P_D(I)\in D\smallsetminus C$ and 
\begin{equation}
(\forall x\in I)\quad
y(x) = \frac{f(d)}{e^{q(d)}}e^{q(x)}.
\end{equation}
The so-constructed function $y\colon\RR\to\RR$ is convex,
and it satisfies $Z=G_y$.
\end{theorem}
\begin{proof}
It is easy to check that \eqref{e:Z1} is the same as \eqref{e:Z}.
Let $x\in\RR$ such that $f(x)>0$ and $\theta(x)\geq 0$, and set
\begin{equation}
\label{e:z}
z(x) = \frac{|f'(x)|}{\sqrt{2L}}-\sqrt{\rho} =
\frac{\sgn\big(f'(x)\big)f'(x)}{\sqrt{2L}}-\sqrt{\rho} =
\sqrt{\theta(x)+\rho}-\sqrt{\rho}\geq 0.
\end{equation}
Then 
\begin{equation}
\label{e:dz}
z'(x) = \frac{\sgn\big(f'(x)\big)f''(x)}{\sqrt{2L}}.
\end{equation}
Using the convexity of $f$, \eqref{e:ftheta}, 
\eqref{e:z}, and \eqref{e:dz}, 
we obtain 
\begin{subequations}
\label{e:angela}
\begin{align}
0 &\leq f''(x)\big(f(x)-\theta(x)\big)\\
&= f''(x)\left( f(x) -
\left(\frac{|f'(x)|}{\sqrt{2L}}+\sqrt{\rho}\right)
\left(\frac{|f'(x)|}{\sqrt{2L}}-\sqrt{\rho}\right)\right)\\
&= f''(x)\left(f(x)+z(x)\left(
z(x)-\frac{2|f'(x)|}{\sqrt{2L}}\right)\right)\\
&=f''(x)\big(f(x)+z^2(x)\big) 
- f'(x)2z(x)\frac{\sgn\big(f'(x)\big)f''(x)}{\sqrt{2L}}\\
&=f''(x)\big(f(x)+z^2(x)\big) - f'(x)\big(2z(x)z'(x)\big). 
\end{align}
\end{subequations}
Because $x-Zx= (f(x)+z^2(x))/f'(x)$ is 
continuous, it is clear that there is an antiderivative $q$ on
$I$ such that
\begin{equation}
\label{e:dq}
q'(x) = \frac{1}{x-Zx} = \frac{f'(x)}{f(x)+z^2(x)}.
\end{equation}
Calculus and \eqref{e:angela} now result in
\begin{subequations}
\label{e:ddq}
\begin{align}
q''(x) &= \frac{f''(x)\big(f(x)+z^2(x)\big) -
f'(x)\big(f'(x)+2z(x)z'(x)\big)}{\big(f(x)+z^2(x)\big)^2}\\
&= \frac{f''(x)\big(f(x)-\theta(x)\big) -
\big(f'(x)\big)^2}{\big(f(x)+z^2(x)\big)^2}.
\end{align}
\end{subequations}
Observe that $y$ is clearly continuous everywhere. 
Furthermore,
$y'(x) = \frac{f(d)}{e^{q(d)}} e^{q(x)}q'(x)$
and hence, using \eqref{e:dq}, \eqref{e:ddq} and again
\eqref{e:angela}, we obtain 
\begin{align}
y''(x) &= \frac{f(d)}{e^{q(d)}}\Big( e^{q(x)}\big(q'(x)\big)^2 +
e^{q(x)}q''(x)\Big)\\
&=  \frac{f(d)}{e^{q(d)}} e^{q(x)}\Big( \big(q'(x)\big)^2 +
q''(x)\Big)\\
&= y(x)
\frac{f''(x)\big(f(x)-\theta(x)\big)}{\big(f(x)+z^2(x)\big)^2}\\
&\geq 0.
\end{align}
Hence $y$ is convex on $I$.
As $x\in I$ approaches $d$,
we deduce (because $d\notin C$, i.e., $f(d)>0$) that 
$q'(x)\to f'(d)(f(d)+z^2(d))^{-1} \to f'(d)/f(d)$
and hence that $y'(x)\to f(d)/e^{q(d)}e^{q(d)}f'(d)/f(d) = f'(d)$.
It follows that $y$ is convex on $\RR$.
Finally, if $x\notin D$, then
$G_y(x) = x - y(x)/y'(x) = x-1/q'(x) = x - (x-Zx)=Zx$.
\end{proof}

\begin{example}
Consider Theorem~\ref{t:caifang}
and assume that $f\colon x\mapsto x^2-1$,
that $L=3$, and that $\rho=1$.
Then \eqref{e:Z1} turns into
\begin{equation}
Zx= 
\begin{cases}
x, &\text{if $|x|\leq 1$;}\\[+3mm]
\displaystyle 
\frac{x^2+1}{2x}, 
&\text{if $1<|x|\leq\sqrt{6}/2$;}\\[+3mm]
\displaystyle \frac{x^2+2\sqrt{6}|x|}{6x}, 
&\text{if $|x|>\sqrt{6}/2$.}
\end{cases}
\end{equation}
Hence $D= \big[-\sqrt{6}/2,\sqrt{6}/2\big]$. 
Using elementary manipulations, we obtain 
\begin{equation}
(\forall x\in \RR\smallsetminus D)\quad
q(x) =
\tfrac{6}{5}\ln\big(\tfrac{5}{6}|x|-\tfrac{\sqrt{6}}{3}\big);
\end{equation}
consequently, the function $y$, given by 
\begin{equation}
(\forall x\in\RR) \quad y(x) = 
\begin{cases}
x^2-1, &\text{if $|x|\leq\sqrt{6}/2$;}\\[+3mm]
\displaystyle\frac{72^{1/5}}{6}\big(5|x|-2\sqrt{6}\big)^{6/5},
&\text{if $|x|>\sqrt{6}/2$,}
\end{cases}
\end{equation}
satisfies $G_y = Z$ by Theorem~\ref{t:caifang}. 
\end{example}

\section*{Acknowledgments}

HHB was partially supported by a Discovery Grant and an Accelerator
Supplement of the Natural Sciences and Engineering
Research Council of Canada (NSERC) and by the Canada Research Chair Program.
CW was partially supported by a grant from
Shanghai Municipal Commission for Science and Technology
(13ZR1455500). 
XW was partially supported by a Discovery Grant of NSERC.
JX was partially supported by NSERC grants of HHB and XW.


\end{document}